\documentclass[]{amsart}
\usepackage{amssymb}
\usepackage{graphicx}
\usepackage{xcolor} 
\usepackage{amsmath}
\usepackage{amsthm}
\usepackage{verbatim}
\usepackage{enumitem}

\newtheorem{theorem}{Theorem}

\newtheorem{lemma}[theorem]{Lemma}

\newtheorem*{theorem*}{Theorem}
\newtheorem*{lemma*}{Lemma}

\theoremstyle{definition}

\newtheorem{remark}{Remark}

\newtheorem*{remark*}{Remark}
\newtheorem*{example*}{Example}
\newtheorem*{er*}{Examples and Remarks}

\newcommand{\nnrm}[1]{{\vert\kern-0.25ex\vert\kern-0.25ex\vert #1 \vert\kern-0.25ex\vert\kern-0.25ex\vert}}

\begin{document}

\title[Local limit theorem]
{A sufficient condition for local limit theorem
} 

\author{Kaoru Yoneda} 
\address{Department of Mathematics, Osaka Metropolitan University, 3-3-138 Sugimoto Sumiyoshi, Osaka, 558-8585, Japan
} 
\email{fdnmj416@yahoo.co.jp} 

\author{Tsuyoshi Yoneda} 
\address{Graduate School of Economics, Hitotsubashi University, 2-1 Naka, Kunitachi, Tokyo 186-8601, Japan} 
\email{t.yoneda@r.hit-u.ac.jp}


\date{\today} 


\begin{abstract}  
We give a sufficient condition for the local limit theorem.
To construct it, we employ 
infinite times of convolutions of probability density functions.
\end{abstract}

\maketitle

\section{Introduction}

There have been a lot of attempts to prove the central limit theorem (CLT) without using characteristic functions (i.e. Fourier transform). See \cite{Ca,Ch} for the recent development of CLT and references therein.
In spite of the difficulty of building
the theory of characteristic functions, in this paper, we focus on local limit theorem, 
 and give a sufficient condition for the local limit theorem.

 To be self-contained,
first we reformulate 
the infinite sum of independent random variables into infinite times of convolutions of probability density functions.
For a random variable $X$, let us define the corresponding monotone increasing function $\mu$ such that
\begin{equation*}
\mu(x):=P(X\leq x).
\end{equation*}
Note that $0\leq \mu(x)\leq 1$ 
and $P(y<X\leq x)=P(X\leq x)-P(X\leq y)$
for $x\in\mathbb{R}$.
By Lebesgue's decomposition theorem,
we can split $\mu$ into three parts:
\begin{equation*}
\mu=\mu^p+\mu^{ac}+\mu^{sc},
\end{equation*}
where $\mu^p$ is the point mass part, $\mu^{ac}$ is the absolute continuous part
and $\mu^{sc}$ is the singular continuous part.
The typical example of the singular continuous part is the Cantor distribution.
Since we focus on the local limit theorem, in what follows, we assume $\mu^p=\mu^{sc}=0$.
Then we obtain the following probability density function (PDF):
\begin{equation*}
\rho^{ac}:=\frac{d\mu^{ac}}{dx}\in L^1(\mathbb{R}),
\end{equation*}
where $\rho^{ac}\geq 0$ with $\int_\mathbb{R}\rho^{ac}(x)dx=1$. 
Let $X_1(x)=x$ and $X_2(x)=x$ ($x\in\mathbb{R}$) be independent random variables, 
and $\rho^{X_1}$ $\rho^{X_2}$ be the corresponding PDFs.
Let $Y=X_1+X_2$, 
and  $\mu^Y$ be the corresponding PDF.
By Benn's diagram, we have 
\begin{equation*}
\begin{split}
&P\left(x-\delta_1/2+\delta_2/2<Y\leq x+\delta_1/2-\delta_2/2\right)\\
\leq & \sum_{j\in\mathbb{Z}}
P\left(x-\delta_2 j-\delta_1/2<X_1\leq x-\delta_2 j+\delta_1/2,\ \delta_2j-\delta_2/2<X_2\leq \delta_2 j+\delta_2/2\right)\\
\leq & P\left(x-\delta_1/2-\delta_2/2<Y\leq x+\delta_1/2+\delta_2/2\right)\\
\end{split}
\end{equation*}
for $\delta_1\gg \delta_2>0$.
On the other hand, by the statistical independence of random variable, we have 
\begin{equation*}
\begin{split}
&
\sum_{j\in\mathbb{Z}}
P\left(x-\delta_2 j-\delta_1/2<X_1\leq x-\delta_2 j+\delta_1/2,\ \delta_2 j-\delta_2/2<X_2\leq \delta_2 j+\delta_2/2\right)\\
=&\sum_{j\in\mathbb{Z}}
P\left(x-\delta_2 j-\delta_1/2<X_1\leq x-\delta_2 j+\delta_1/2)P(\delta_2j-\delta_2/2<X_2\leq \delta_2 j+\delta_2/2\right).\\
\end{split}
\end{equation*}
Then, taking $\delta_1,\delta_2\to 0$, we have 
\begin{equation*}
\rho^Y=\rho^{X_1}\ast\rho^{X_2}\quad a.e.
\end{equation*}
Repeating this argument, for $Y=\sum_{j=1}^nX_j$, we have 
\begin{equation}\label{convolution}
\rho^Y=\rho^{X_1}\ast\rho^{X_2}\ast\cdots\ast\rho^{X_n}.
\end{equation}

\begin{remark}
The formula of \eqref{convolution}, that is, infinite times of convolutions have already been studied, in the different context.
 The second author \cite{Y1,Y2} constructed special smooth functions
which are compactly supported, and are satisfying the corresponding functional-differential equations of advanced type.
To construct such special functions, he employed the formula \eqref{convolution} with various scalings.
See also \cite{RR,V}.
\end{remark}

\section{Sufficient condition for local limit theorem}

In this section, we give a sufficient condition for 
the local limit theorem.
For any fixed $\sigma\in\mathbb{Z}_{\geq 1}$,
we set $h_N:=\sqrt{\sigma/N}$ ($N=\sigma,\sigma+1,\cdots$).
Let $\rho\in L^1(\mathbb R)$ be such that 
\begin{equation}\label{integrable condition}
\rho\geq 0,\quad \int_\mathbb{R} x\rho(x)dx=0,\quad
\int_{-\infty}^\infty x^2\rho(x)dx=:\gamma<\infty,
\end{equation}
where this $\sigma_\pm$ means $\sigma\pm \delta$ for some sufficiently small $\delta>0$.
Note that $\hat \rho\in C^2(\mathbb{R})$ due to \eqref{integrable condition}.
Applying the Taylor expansion to $(\hat\rho)^2$, we see that 
\begin{equation*}
 |\hat\rho(t)|^{\frac{1}{t^2}}=((\hat\rho(t))^2)^{1/2t^2}
=\left(1+\frac{(\hat\rho^2)''(\theta t) t^2}{2!}
\right)^{\frac{1}{2t^2}}.
\end{equation*}
Since
\begin{equation*}
(\hat\rho^2)''(\theta t)|_{t=0}=2\hat \rho(0)\hat\rho''(0)=-2\gamma,
\end{equation*}
 we have
\begin{equation*}
\begin{split}
1+\frac{(\hat\rho^2)''(\theta t) t^2}{2!}
&\leq 1-\gamma_-t^2\quad (0<t<T)\\
\end{split}
\end{equation*}
for some $\theta\in (0,1)$ and $T(<\sqrt e)$.

Thus we have 
\begin{equation*}
|\hat \rho(t)|^{\frac{1}{t^2}}\leq \left(1-\gamma_-t^2\right)^{\frac{1}{\gamma_-t^2}
\frac{\gamma_-}{2}}
\to e^{-\frac{\gamma_-}{2}}<1\quad (t\to 0+).
\end{equation*}

\begin{remark}
The function $(1-s)^{1/s}$ ($s=t^2\gamma_-$) is monotone decreasing, thus 
\begin{equation}\label{max-value}
|\hat \rho(t)|^{\frac{1}{t^2}}\leq e^{-\frac{\gamma_-}{2}}<1\quad\text{for}\quad 0<t<T.
\end{equation}
This is due to the fact that  
\begin{equation*}
(1-s)^{\frac{1}{s}}=\exp\left(\frac{1}{s}\log(1-s)\right)
\end{equation*}
and
\begin{equation*}
\begin{split}
\frac{d}{ds}\left(\frac{1}{s}\log(1-s)\right)
&=-\frac{1}{s}\left(\frac{1}{1-s}+\frac{1}{s}\log(1-s)\right)\\
&=
-\frac{1}{s}\left(\frac{1}{2}s+\left(1-\frac{1}{3}\right)s^2+\left(1-\frac{1}{4}\right)s^3+\cdots\right)<0.
\end{split}
\end{equation*}
\end{remark}
With the aid of the above observation, in this paper we also assume the following:
\begin{equation}\label{frequency condition}
 |\hat \rho(t)|\leq \frac{C}{|t|^{\frac{1}{\sigma_-}}}
\quad\text{for}\quad T<t<\infty,
\end{equation}
where $C=|T|^{\frac{1}{\sigma_-}}(1-\gamma_-T^2)$.
For $\ell>0$, let $\rho_\ell(x):=\ell^{-1}\rho(\ell^{-1} x)$ and let
\begin{equation*}
\Phi_N(x):=\underbrace{\rho_{h_N}\ast\rho_{h_N}\ast\cdots\ast \rho_{h_N}}_{N}(x).
\end{equation*}
Then we can state the main theorem as follows:
\begin{theorem}
Assume $\rho\in L^1(\mathbb{R})$ satisfies
 \eqref{integrable condition} and \eqref{frequency condition}.
Then we have 
\begin{equation*}
\lim_{N\to\infty}\Phi_N(x)=\sqrt{\frac{2\pi}{\gamma}}\exp\left(-\frac{x^2}{2\gamma}\right)\quad \text{uniformly}.
\end{equation*}
\end{theorem}
To prove this local limit theorem, it is enough to show the following lemma,
due to Bernstein's inequality. 
\begin{lemma}\label{key lemma}
We have 
\begin{equation*}
\lim_{N\to\infty}\|\hat\Phi(y)-e^{-\gamma y^2/2}\|_{L^1_y(\mathbb{R})}=0.
\end{equation*}
\end{lemma}
\begin{proof}
To prove this key lemma, we just apply Lebesgue's convergence theorem.
Thus it is enough to show the following two assertions:
\begin{equation}\label{1}
\lim_{N\to\infty}\hat \Phi_N(y)=e^{-\gamma y^2/2}\quad \text{for any}\quad y\in\mathbb{R}
\end{equation}
and there exists $M(y)\in L^1_y(\mathbb{R})$ such that 
\begin{equation}\label{2}
|\hat\Phi_N(y)|\leq
\sup_{N\geq \sigma}\left|\left(\hat \rho(h_Ny)\right)^N\right|\leq M(y).
\end{equation}
First we show \eqref{1}.
We see that 
\begin{equation}\label{sinc-based}
\left(\hat \rho(h_N y)\right)^N
=
\left(\hat\rho(h_N y)\right)^{\left(\frac{1}{h_Ny}\right)^2\sigma y^2}
=\left(\hat\rho(t)\right)^{\frac{\sigma y^2}{t^2}},
\end{equation}
where $t:=h_Ny$. Thus, to show the pointwise convergence for any $y\in\mathbb{R}$, 
it is enough to show that the complex function
\begin{equation*}
\left(\hat\rho(t)\right)^{\frac{1}{t^2}}
\end{equation*}
converges for $t\to 0+$.
By the assumption $\int_\mathbb{R}x\rho(x)dx=0$, we have 
\begin{equation*}
\begin{split}
\left(\hat\rho(t)\right)^{1/t^2}
&=\left(1+\frac{\hat\rho''(\theta t) t^2}{2!}
\right)^{1/t^2}
=
\exp\left(\frac{1}{t^2}\log\left(1+\frac{\hat\rho''(\theta t) t^2}{2!}
\right)\right)\\
&=
\exp\left(\frac{1}{t^2}\left(\frac{\hat\rho''(\theta t) t^2}{2!}+O(t^3)
\right)\right)\\
&=
\exp\left(\frac{\hat\rho''(\theta t)}{2}+O(t)\right)
\end{split}
\end{equation*}
for some $\theta\in (0,1)$ depending on $t$.
Thus 
\begin{equation*}
\lim_{t\to 0+}\left(\hat\rho(t)\right)^{1/t^2}=e^{-\gamma/2}.
\end{equation*}
This gives \eqref{1}. Next we show \eqref{2}.
By \eqref{sinc-based},
 it is enough to show 
\begin{equation*}
\left(\sup_{0<t\leq y}
 |\hat\rho(t)|^{\frac{1}{t^2}}
\right)^{\sigma y^2}\in L^1_y(\mathbb{R}_{\geq 0}).
\end{equation*}
This is due to the fact that 
\begin{equation*}
N\geq \sigma\Leftrightarrow \sqrt{N/\sigma}\geq 1\Leftrightarrow y\geq t.
\end{equation*}
For $t\in (T,\infty)$, we see that
\begin{equation*}
\left|\hat\rho(t)\right|^{\frac{1}{t^2}}\leq \left|\frac{C}{t}\right|^{\frac{1}{\sigma_-t^2}}=\exp\left(-(1/\sigma_-t^2)\log (t/C)\right)=:\zeta(t)
\end{equation*}
by \eqref{frequency condition}
and 
\begin{equation*}
-\frac{d}{dt}\left(\frac{1}{\sigma_-t^2}\log (t/C)\right)=-\frac{C}{\sigma_-t^3}(1-2\log t).
\end{equation*}
Thus this function $\zeta$ is monotone decreasing up to $t=\sqrt e$,
and monotone increasing from there.
Note that 
\begin{equation*}
\lim_{t\to\infty}\zeta(t)
=
\lim_{t\to\infty}\exp\left(-\frac{\log (t/C)}{\sigma_-t^2}\right)=1
\end{equation*}
and then we can find a unique $y_0\in(\sqrt e,\infty)$ satisfying  
$\zeta(y_0)=e^{-\frac{\gamma_-}{2}}$.
Therefore
\begin{equation*}
\left(\sup_{0<t<y} |\hat\rho(t)|^{\frac{1}{t^2}}\right)^{\sigma y^2}\leq 
e^{-\frac{\gamma_-}{2}}
\quad\text{for}\quad y\in[0, y_0]
\end{equation*}
and
\begin{equation*}
\left(\sup_{0<t< y}
 |\hat\rho(t)|^{\frac{1}{t^2}}
\right)^{\sigma y^2}
\leq \left(\left|\frac{C}{y}\right|^{\frac{1}{\sigma_-y^2}}\right)^{\sigma y^2}\leq Cy^{-(1_+)}
\quad\text{for}\quad y\in (y_0,\infty),
\end{equation*}
and this means there exists a 
 function $M(y)\in L^1_y(\mathbb{R})$ satisfying 
\begin{equation*}
M(y)=
\begin{cases}
e^{-\frac{\gamma_-}{2}}\quad (y\leq y_0),\\
Cy^{-(1_+)}\quad (y_0< y).
\end{cases}
\end{equation*}
This is the desired function satisfying \eqref{2} and 
we complete the proof of the Lemma \ref{key lemma}.
\end{proof}

\vspace{0.5cm}
\noindent
{\bf Acknowledgments.}\ 
We are grateful to Professor Katusi Fukuyama for valuable comments. 
This work was done while KY is a professor emeritus of Osaka Prefecture University
(now Osaka Metropolitan University).
Research of  TY  was partly supported by the JSPS Grants-in-Aid for Scientific
Research  24H00186.
This paper was a part of the lecture note on the class:
Analysis II (winter semester 2023,2024) for undergrad course in Hitotsubashi University.

\bibliographystyle{amsplain} 

\end{document}